\documentclass[11pt]{amsart}
\usepackage{amsfonts}
\usepackage{amssymb}

\usepackage[T1]{fontenc}
\usepackage{amsmath}
\usepackage{dsfont}
\usepackage{color}
\usepackage{graphicx}

\newtheorem{theorem}{Theorem}[section]
\newtheorem{lemma}[theorem]{Lemma}

\theoremstyle{definition}
\newtheorem{df}{Definition}
\newtheorem{example}[df]{Example}

\newcommand{\N}{\mathbb N}

\newcommand{\R}{\mathbb R}

\newcommand{\Rj}{\mathcal{R}}
\newcommand{\G}{\mathcal{G}}

\newcommand{\Lj}{\mathcal{L}}
\newcommand{\ha}{\symbol{94}}
\makeatletter
\@namedef{subjclassname@2020}{%
  \textup{2020} Mathematics Subject Classification}
\makeatother


\begin{document}

\author{Piotr Nowakowski}
\address{Faculty of Mathematics and Computer Science, University of Lodz,
ul. Banacha 22, 90-238 \L \'{o}d\'{z}, Poland
\\
ORCID: 0000-0002-3655-4991}
\email{piotr.nowakowski@wmii.uni.lodz.pl}
\title{On a new condition implying that an achievement set is a Cantorval and its applications}
\date{}
 
\begin{abstract}
Given a nonincreasing sequence of positive numbers $(a_n)$ such that the series $\sum a_n$ is convergent, by $E(a_n)$ we denote the set of all subsums of the series $\sum a_n$ and call it the achievement set of $(a_n)$. It is well known that such a set can be a finite union of closed intervals, a Cantor set or a Cantorval. We give a new condition implying that the last possibility occurs. We also show how we can use this condition to produce new achievable Cantorvals. In particular, we prove that Kakeya conditions cannot tell us more about the form of the achievement set than it was proved by Kakeya.

\end{abstract}
\maketitle
\section{Introduction}

Let $(a_n)$ be a nonincreasing sequence of positive numbers such that the series $\sum a_n$ is convergent. We define the achievement set of $\sum a_n$ (or $(a_n)$) as the set of all its subsums, that is,
$$E(a_n):=\{\sum_{n\in A}a_n\colon A\subset \N\}.$$
By $r_n$ we denote the $n$-th remainder of the series, that is, $r_n:= \sum_{i=n+1}^\infty a_i$. 

The studies on achievement sets started over 100 years ago by Kakeya. In his paper \cite{Kakeya} he proved that $E(a_n)$ is an interval if and only if $a_n \leq r_n$ for all $n \in \N$ and it is a finite union of intervals if and only if $a_n \leq r_n$ for almost all $n \in \N$. Moreover, if $a_n > r_n$ for all $n \in \N$, then $E(a_n)$ is a Cantor set (actually it suffices if $a_n > r_n$ for almost all $n$). He conjectured that only these two forms are possible. In particular, he believed that $E(a_n)$ is a Cantor set if $a_n> r_n$ for infinitely many $n$. It took about 60 years to disprove the conjecture. Eventually, Guthrie and Nymann in \cite{GN88} gave full classification of achievement sets, showing that there is one more possible form: a Cantorval. 

We say that a compact and perfect set $D \subset \R$ 
is a Cantorval if each endpoint of every gap of $D$ (a component of the complement of $D$) is an accumulation point of other gaps as well as of nontrivial components of $D$ (compare \cite{MO}).

Finding conditions which imply that an achievement set is a Cantor set or a Cantorval became one of the most important problems in research concerning achievement set (see e.g. \cite{AI,BBFS,recover,BFPW2,BFS,FN23,Jones,KMV,VMPS19}).  Despite so many papers on this topic, we are still very far away from the full characterization, when the given achievement set is a Cantor set or a Cantorval.

From Kakeya's result we know what happens if one of the inequalities $a_n\leq r_n$ or $a_n>r_n$ holds for almost all $n$. It is quite natural to ask what happens if although both inequalities hold infinitely many times, but one of them holds on a small set of indices (in some sense). In the paper \cite{MM} the authors showed that for any infinite set $K \subset \N$ there exists a nonincreasing sequence $(a_n)$ such that $\{n\in \N\colon a_n > r_n\} = K$ and $E(a_n)$ is a Cantor set. On the other hand, they also showed that for any $\alpha \in [0,1)$, there exists a nonincreasing sequence $(a_n)$ such that $\{n\in \N\colon a_n > r_n\}$ has asymptotic density equal to $\alpha$ and $E(a_n)$ is a Cantorval. In the paper \cite{PP} the same assertion was obtained for $\alpha =1$. In \cite{MM} there was posed a problem if for any infinite set $K \subset \N$ there exists a nonincreasing sequence $(a_n)$ such that $\{n\in \N\colon a_n > r_n\} = K$ and $E(a_n)$ is a Cantorval. We will solve this problem later in the paper.  

One class of sequences is especially important for research concerning achievement sets, namely multigeometric sequences, that is, sequences of the form $(k_1q,k_2q,\dots,k_mq,k_1q^2,k_2q^2,\dots,k_mq^2,k_1q^3,\dots)$ for some $m \in \N$, $k_1,\dots,k_m \in \R$, $q\in (0,1)$. Achievement sets of multigeometric sequences are fractals. Important question for such sets is a variant of Palis conjecture: is it true that every achievement set of multigeometric sequence is either of Lebesgue measure zero or has nonempty interior. Though the results in this paper do not answer this question, they may help in the future to solve this problem (which will be justified in section 4).

The paper is organized as follows. In the next section we introduce definitions, notation and basic lemmas we will use throughout the article. In the 3. section we will prove the main theorem of the paper. The last section will consist of examples and applications of the main theorem.

\section{Preliminaries}

Let $I$ be an interval. By $l(I)$ we denote its left endpoint and by $r(I)$ its right endpoint. For two sequences $t,s$ by $t\ha s$ we will denote their concatenation. By $i^{(n)}$ we will denote a sequence consisting of $n$ $i$-s (that is, a the sequence $i\ha i\ha\ldots \ha i$). By $|t|$ we denote the length of a sequence $t$. By $A^{<k}$ we denote the set of all sequences of length $j$, for $j\in \{0,1,\dots,k-1\}$ with all terms in $A$. 

Whenever $m<k$, then we assume that $\sum_{i=k}^m a_i = 0$ for any sequence $(a_i)$.

Throughout the paper we will always use the assumptions and notation given below.
\begin{df} \label{as}
Let $a=(a_n)$ be a decreasing sequence of positive terms such that $\sum_n a_n$ is convergent, $a_n > r_n$ infinitely many times, $a_n < r_n$ infinitely many times and $a_1<r_1$.
We are now going to inductively construct two sequences $(k_n)$ and $(m_n)$.
Let $k_1 =  \min\{i \in \N\colon a_i > r_i\}$ and  let $m_1 = \min\{i \in \N\colon a_{k_1+i} \leq r_{k_1+i}\}$. If we have defined $k_n$ and $m_n$ for some $n\in \N$, then let $k_{n+1} = \min\{i > k_n+m_n \colon a_i > r_i\}$ and let $m_{n+1} = \min\{i \in \N\colon a_{k_{n+1}+i} \leq r_{k_{n+1}+i}\}$. Put $K_n:=\{k_{n},\dots, k_{n}+m_n-1\}$ and $K:=\bigcup_{n=1}^\infty K_n.$
\end{df}

Let $I:=I_\emptyset = [0,r_0]$.
Let $I_0 = [0,r_1]$ and $I_1 = [a_1, r_0]$.
Let $n\in \N$ and $t \in \{0,1\}^n$. Assume that we have defined the interval $I_t$. Then we define $I_{t\ha 0} = [l(I_t),l(I_t)+r_{n+1}]$ and  $I_{t\ha 1} = [l(I_t)+a_{n+1}, r(I_t)].$ If $a_{n+1} > r_{n+1}$, then denote
by $G_t$ a gap in $I_t$, that is, the interval $(l(I_t)+r_{n+1},l(I_t)+a_{n+1})=(r(I_{t\ha 0}),l(I_{t\ha 1}))$. We call it a gap of order $n+1$ (note that $G_t$ does not have to be a gap in the whole set $E(x)$).
If $a_{n+1} \leq r_{n+1}$, then denote
by $Z_t$ an overlap in $I_t$, that is, the interval $[l(I_t)+a_{n+1},l(I_t)+r_{n+1}]= [l(I_{t\ha 1}),r(I_{t\ha 0})]$.
It is well known that
$$E(x)=\bigcap_{n\in\N} \bigcup_{t\in\{0,1\}^n} I_t.$$

For every $n \in \N$ we say that a gap $G_{0^{(k_n-1)}}$ comes from $0$ and a gap $G_{1^{(k_n-1)}}$ comes from $1$. 
Let $t\in \{0,1\}^k$ for $k \in \N$. For every $n \in \N$ such that $k_n-1>k$ we say that a gap $G_{t\ha 0^{(k_n-1-k)}}$ comes from $l(I_t)$ and a gap $G_{t\ha 1^{(k_n-1-k)}}$ comes from $r(I_t)$.
If $k+1 \in K$, then we say that a gap comes from $G_t$ if it comes from $l(G_t)=r(I_{t\ha 0})$ or from $r(G_t) = l(I_{t \ha 1})$. 
If $|t|=k_n-1$ for some $n \in \N$ and $m_n>0$, then the gaps $G_{t\ha t_{k_n-1} \ha s}$, where  $s\in \{0,1\}^{<m_n-1}$ are called neighbours of $G_t$.  

For $n \in \N$ let
$$
\mathcal{L}_n\ =\ \begin{cases}
\{G_{0^{(k_n-1)}}\}\hspace{.5in}&\text{if $m_n=1$}\\
\{G_{0^{(k_n-1)}},\, G_{0^{(k_n)}}\}&\text{if $m_n=2$}\\
\{G_{0^{(k_n-1)}},\, G_{0^{(k_n)}}\}\cup\{G_{0^{(k_n)\wedge} s}:\ |s|\le m_n-2\,\} &\text{if $m_n\ge 3$}
\end{cases}
$$
and
$$
\mathcal{R}_n\ =\ \begin{cases}
\{G_{1^{(k_n-1)}}\}\hspace{.5in}&\text{if $m_n=1$}\\
\{G_{1^{(k_n-1)}},\, G_{1^{(k_n)}}\}&\text{if $m_n=2$}\\
\{G_{1^{(k_n-1)}},\, G_{1^{(k_n)}}\}\cup\{G_{1^{(k_n)\wedge} s}:\ |s|\le m_n-2\,\} &\text{if $m_n\ge 3$}
\end{cases}.
$$
Thus, $\Lj_n$ consists of the gap which comes from $0$ of order $k_n$ and all its neighbours, and $\Rj_n$ consists of the gap which comes from $0$ of order $k_n$ and all its neighbours. 

We will now define inductively families of gaps $(\G_i)$. We assume that $k_1 > 1$. 
$\G_1$ consists of gaps coming from $0$ and $1$ of order $k_1$ and all their neighbours, that is, $\G_1:=\Lj_1 \cup \Rj_1$. 

Figure \ref{rysG} shows the set $I \setminus \bigcup\G_1$ for $m_1=3$. The letters correspond to the following gaps:
$A = G_{0^{(k_1+1)}}$, $B = G_{0^{(k_1)}}$, $C = G_{0^{(k_1)}\ha 1}$, $D = G_{0^{(k_1-1)}}$, $E = G_{1^{(k_1-1)}}$, $F = G_{1^{(k_1)}\ha 0}$, $G = G_{1^{(k_1)}}$, $H = G_{1^{(k_1+1)}}$. We have $\Lj_1 = \{A,B,C,D\}$ and $\Rj_1=\{E,F,G,H\}$. The gap $D$ comes from $0$, and $E$ comes from $1$. Gaps $A,B,C$ are the neighbours of $D$ and $F,G,H$ are neighbours of $E$.
 \begin{figure} [h] 
\includegraphics[width=0.9\textwidth]{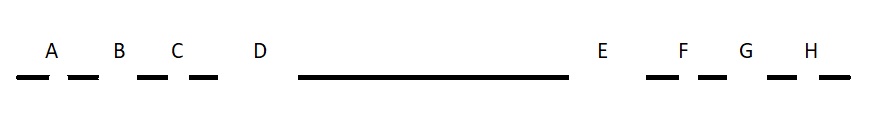}
\caption{}
\label{rysG}
\end{figure}

Assume that for some $n\in \N$ we have defined a family of gaps $\G_n$. In the family $\G_{n+1}$ there are all gaps of order $k_{n+1}$ which comes from gaps from $\bigcup_{i=1}^n \G_i$ or from $0$, or from $1$ and their neighbours. 
Formally,

$$\mathcal{G}_{n+1}=  \mathcal{L}_{n+1}\,\cup\,\mathcal{R}_{n+1}\,\cup\,\bigcup_{\substack{P\in\bigcup_{i=1}^n \G_i\\ L\in\mathcal{L}_{n+1}}}(L+\max P)\,\cup\, \bigcup_{\substack{P\in\bigcup_{i=1}^n \G_i\\ R\in\mathcal{R}_{n+1}}}(R+\min P-r_0).$$
Let $\G = \bigcup_{n\in \N} \G_n$.

%

For intervals $I,J$ we write that $I<J$ if $r(I)<l(J)$. Let $\mathcal{H}$ be a family of gaps.
We say that gaps $G<H$ are consecutive gaps in $\mathcal{H}$ if there is no gap $J \in \mathcal{H}$ such that $G<J<H$.

Before we prove the main theorem we will write in one place properties we will frequently use during the proof. They immediately follow from the definitions of intervals $I_t, G_t$ and $Z_t$. 
\begin{lemma}\label{lem}
Let $(a_n)$ be a sequence as in Definition \ref{as}. Let $k \in \N$, $t\in \{0,1\}^k$, $m\in \N$ and $j\in \N$ be such that $k+j \in K$.
Then
\begin{itemize}
\item[(i)] If $k+1 \in K$, then $l(G_t) = r(I_{t\ha 0})$, $r(G_t) = l(I_{t\ha 1}).$

\item[(ii)] If $k+1 \notin K$, then $l(Z_t) = l(I_{t\ha 1})$, $r(Z_t) = r(I_{t\ha 0}).$
 
\item[(iii)] $l(I_{t\ha 0^{(m)}}) = l(I_t)$, $r(I_{t\ha 1^{(m)}}) = r(I_t)$.

\item[(iv)] $r(G_{t\ha 0^{(j-1)}}) - l(I_t) = a_{k+j}$, $l(G_{t\ha 0^{(j-1)}}) - l(I_t) = r_{k+j},$
$r(I_t)-l(G_{t\ha 1^{(j-1)}}) = a_{k+j},$
$r(I_t)-r(G_{t\ha 1^{(j-1)}}) = r_{k+j}.$
\end{itemize}

\end{lemma}

\section{Main result}
Before we formulate the main result let us introduce some procedure. 
\begin{df}{\textbf{Star Procedure}}
Let $(a_n)$ be a sequence as in Definition \ref{as}. For $i \notin K$ let $\delta_i = r_i - a_i$. 
For $i<k_1-1$ let $M_i^1:= \delta_i$. Assume that for some $n \in \N$ we have defined a finite sequence $(M^n_i)_i$ of positive numbers. 
Consider the
following conditions depending on numbers $M^n_i$ and $\delta_{k_n-1}$:
$$(*)\,\,\,  \left\{ \begin{array}{ccc}
M^n_i - a_{k_n} >0 \\ 
r_{k_n}+r_{k_n+m_n-1} -M^n_i >0 \\ 
\forall_{j \in \{1,\dots, m_n-1\}} \,\, M^n_i -2a_{k_n+j} - \sum_{l=1}^{j-1} a_{k_n+l}>0;  
\end{array}%
\right. $$
$$(**)\,\,\, M^n_i-2a_{k_n} > 0;$$
$$(*')\,\,\,  \left\{ \begin{array}{ccc}
\delta_{k_n-1} - \sum_{j=1}^{m_n-1} a_{k_n+j} >0 \\ 
\forall_{j \in \{1,\dots, m_n-1\}} \,\, 2r_{k_n+j}-\delta_{k_n-1}+\sum_{i=1}^{j-1}a_{k_n+i} >0. 
\end{array}%
\right. $$
If for every number $M^n_i$ either condition $(*)$ or $(**)$ is satisfied and also $(*')$ holds, then we continue the procedure, defining a finite sequence $(M^{n+1}_i)$ of positive numbers as a sequence consisting of left hand sides of all satisfied inequalities from conditions $(*)$, $(*')$ or $(**)$ and all numbers $\delta_i$ for $k_n+m_n\leq i < k_{n+1}-1$.
Otherwise, the procedure breaks.
\end{df}

Now, we are ready to prove the main result of the paper.
\begin{theorem}\label{main}
Let $(a_n)$ be a sequence as in Definition \ref{as}. 
If for every $n \in \N$ we can define positive numbers $(M^n_i)$ as in Star Procedure (that is, Star Procedure never breaks), then $E(a_n)$ is a Cantorval.
\end{theorem}
\begin{proof}

First, for every overlap $Z_t$ we will define some family of gaps $Z_t^*$.
Assume that $n \in\N$, $t \in \{0,1\}^k$, where $k_{n-1}+m_{n-1}-1\leq k \leq k_n-3$ or $k=k_{n-1}-2$, provided $m_{n-1}=1$. We will call $Z_t$ an overlap of the first type. In $Z_t^*$ we have gaps coming from $l(Z_t)=l(I_{t\ha 1})$ and $r(Z_t)=r(I_{t\ha 0})$ and their neighbours. Formally, $G_{t\ha (1-j) \ha j^{(k_n-k-2)}} \in Z_t^*$ for $j\in \{0,1\}$. If $m_n > 1$, then also 
$G_{t\ha (1-j) \ha j^{(k_n-k-1)} \ha s} \in Z_t^*$ for $j\in \{0,1\}$ and $s \in \{0,1\}^{<m_n-1}$.

Now, assume that $m_n > 1$ and $t\in \{0,1\}^{k_n-2}$. We will call $Z_t$ an overlap of the second type. Then $G_{t\ha (1-j) \ha j \ha s} \in Z_t^*$ for $j\in \{0,1\}$ and $s \in \{0,1\}^{<m_n-1}$.

Assume that we have defined all gaps of order at most $k_j+m_j-1$ which are in $Z_t^*$ (of any type). Then in $Z_t^*$ there are also all gaps of order $k_{j+1}$ (and their neighbours) which comes from $l(Z_t)$, $r(Z_t)$ or any gap from $Z_t^*$ of order at most $k_j+m_j-1$.

Let $k \notin K$ and $t \in \{0,1\}^k$. We will show that if $G,H \in Z_t^*,$ then $G,H\subset Z_t$ and $G \cap H = \emptyset$.
Observe that $|Z_t| = r_k-a_k = \delta_k$. If $Z_t$ is of the first type, then, by the assumption, condition $(*)$ or $(**)$ is satisfied for $\delta_k$ and if $Z_t$ is of the second type, then condition $(*')$ holds for $\delta_k$. Consider the cases.

{\bf Case 1.} $\delta_k$ satisfies $(*)$. 
We will show that gaps are placed like on Figure \ref{rys*} (for $m_n=3$) which presents an overlap $Z_t$ and distances between consecutive gaps (or endpoints of $Z_t$ and the nearest gaps) are equal to numbers from the sequence $(M^{n+1}_i)$. The letters correspond to the following gaps: $A=G_{t\ha 0\ha 1^{(k_n-k-2)}}$, $B=G_{t\ha 1\ha 0^{(k_n-k)}}$,
$C=G_{t\ha 0\ha 1^{(k_n-k-1)}\ha 0}$, $D=G_{t\ha 1\ha 0^{(k_n-k-1)}}$, $E=G_{t\ha 0\ha 1^{(k_n-k-1)}}$, $F=G_{t\ha 1\ha 0^{(k_n-k-1)}\ha 1}$, $G=G_{t\ha 0\ha 1^{(k_n-k)}}$, $H=G_{t\ha 1\ha 0^{(k_n-k-2)}}$.

 \begin{figure} [h] 
\includegraphics[width=0.9\textwidth]{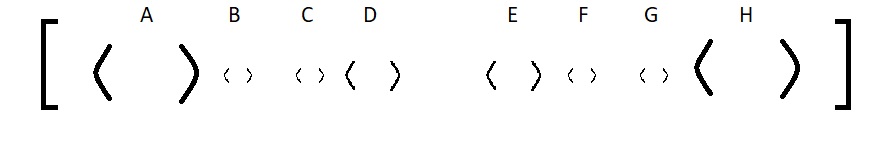}
\caption{}
\label{rys*}
\end{figure}

We have
$G_{t\ha 1 \ha 0^{(k_n-k-2)}} \subset J_{t \ha 1}$, so $l(G_{t\ha 1 \ha 0^{(k_n-k-2)}}) > l(I_{t \ha 1}) = l(Z_t)$. Moreover, by Lemma \ref{lem}, $$r(Z_t) - r(G_{t\ha 1 \ha 0^{(k_n-k-2)}}) = r(Z_t)-l(Z_t) - (r(G_{t\ha 1 \ha 0^{(k_n-k-2)}}) - l(I_{t\ha 1}))= |Z_t| - a_{k_n} = \delta_k - a_{k_n} > 0,$$ so
$G_{t\ha 1 \ha 0^{(k_n-k-2)}} \subset Z_t$. Similarly, $G_{t\ha 0 \ha 1^{(k_n-k-2)}} \subset Z_t$. By $(*)$, we also have
$$l(G_{t\ha 1 \ha 0^{(k_n-k-2)}}) - r(G_{t\ha 0 \ha 1^{(k_n-k-2)}}) = l(Z_t) + r_{k_n} - (r(Z_t)-r_{k_n}) =2r_{k_n} - |Z_t| = 2r_{k_n} - \delta_k \geq r_{k_n} +r_{k_n+m_n-1} - \delta_k >0.$$
So, $G_{t\ha 1 \ha 0^{(k_n-k-2)}} \cap G_{t\ha 0 \ha 1^{(k_n-k-2)}} = \emptyset$.

In particular, $G_{t\ha 0 \ha 1^{(k_n-k-2)}} \subset I_{t\ha 1 \ha 0^{(k_n-k-1)}}$ and $G_{t\ha 1 \ha 0^{(k_n-k-2)}} \subset I_{t\ha 0 \ha 1^{(k_n-k-1)}}$.
 
If $m_{n} >1$, then gaps $G_{t\ha 0 \ha 1^{(k_n-k-1)}}, G_{t\ha 1 \ha 0^{(k_n-k-1)}} \in Z_t^*$.
We have, by $(*)$,
$$l(G_{t\ha 1 \ha 0^{(k_n-k-1)}}) - r(G_{t\ha 0 \ha 1^{(k_n-k-2)}}) = l(Z_t) + r_{k_n+1} - (r(Z_t)-r_{k_n}) = -|Z_t| + r_{k_n}+r_{k_n+1} = -\delta_{k} + r_{k_n}+r_{k_n+1} $$$$\geq -\delta_{k} + r_{k_n}+r_{k_n+m_n-1}>0.$$
Similarly, 
$$l(G_{t\ha 1 \ha 0^{(k_n-k-2)}}) - r(G_{t\ha 0 \ha 1^{(k_n-k-1)}})  = -\delta_{k} + r_{k_n}+r_{k_n+1} >0.$$
Moreover,
$$l(G_{t\ha 0 \ha 1^{(k_n-k-1)}}) - r(G_{t\ha 1 \ha 0^{(k_n-k-1)}}) = r(Z_t) - a_{k_n+1} - (l(Z_t)+a_{k_n+1}) = |Z_t| -2a_{k_n+1} = \delta_{k}-2a_{k_n+1}>0.$$
Thus, $$G_{t\ha 0 \ha 1^{(k_n-k-2)}}<G_{t\ha 1 \ha 0^{(k_n-k-1)}}<G_{t\ha 0 \ha 1^{(k_n-k-1)}}<G_{t\ha 1 \ha 0^{(k_n-k-2)}}.$$
In particular, $G_{t\ha 0 \ha 1^{(k_n-k-1)}} \subset I_{t\ha 1 \ha 0^{(k_n-k-1)}\ha 1}$ and $G_{t\ha 1 \ha 0^{(k_n-k-1)}} \subset I_{t\ha 0 \ha 1^{(k_n-k-1)}\ha 0}$.

Assume that $G,H,G_{v}, G_{w} \in Z_t^*$, where $|v|=|w| = k_n+i$ for $0\leq i \leq m_n-3$ are such that $G<G_v<G_w<H$ are consecutive gaps from $Z_{t}^*$ of order at most $k_n+i+1$, $l(G_v)-r(G)=l(H)-r(G_w) = -\delta_{k} + r_{k_n}+r_{k_n+i+1}$ and $l(G_w)-r(G_v) = \delta_k - 2a_{k_n+i+1}-\sum_{l=1}^i a_{k_n+l}$. 
Again using $(*)$, we get
$$l(G_{v\ha 0})-r(G) = l(G_v)-a_{k_n+i+2}-r(G) = -\delta_{k} + r_{k_n}+r_{k_n+i+1} -a_{k_n+i+2} = -\delta_{k} + r_{k_n}+r_{k_n+i+2}$$$$\geq -\delta_{k} + r_{k_n}+r_{k_n+m_n-1} >0.$$
Analogously, 
$$l(G_v)-r(G_{w\ha 0}) = l(G_{v\ha 1}) - r(G_w) = l(H) - r(G_{w\ha 1}) = -\delta_{k} + r_{k_n}+r_{k_n+i+2} > 0.$$
Moreover,
$$l(G_w)-r(G) = l(G_w)-r(G_v)+r(G_v)-l(G_v)+l(G_v)-r(G) $$$$=  \delta_k - 2a_{k_n+i+1}-\sum_{l=1}^i a_{k_n+l} + a_{k_{n}+i+1}-r_{k_n+i+1} -\delta_{k} + r_{k_n}+r_{k_n+i+1}=r_{k_n}-\sum_{l=1}^{i+1} a_{k_n+l}=r_{k_n+i+1},$$
so $r(G) = l(I_{w\ha 0}).$ 
Therefore,
$$l(G_{w\ha 0})-r(G_{v\ha 0}) = r(G)+r_{k_n+i+2} - (l(G_v)-r_{k_n+i+2}) =  \delta_{k} - r_{k_n}-r_{k_n+i+1}+2r_{k_n+i+2} = \delta_k-2a_{k_n+i+2}-\sum_{l=1}^{i+1}a_{k_n+l}>0.$$
Similarly, $$l(G_{w\ha 1})-r(G_{v\ha 1}) = \delta_k-2a_{k_n+i+2}-\sum_{l=1}^{i+1}a_{k_n+l}>0.$$ 

Finally, we get that gaps $G<G_{v\ha 0}<G_{w\ha 0} < G_v$ and $G_w<G_{v\ha 1}<G_{w\ha 1} < H$ are consecutive gaps in $Z_t^*$ of order at most $k_n+i+2$. In particular, $G_{v \ha 0} \subset I_{w\ha 0\ha 0}$, $G_{w \ha 0} \subset I_{v\ha 0\ha 1}$, $G_{v \ha 1} \subset I_{w\ha 1\ha 0}$ and $G_{w \ha 1} \subset I_{w\ha 1\ha 1}$. By induction, all gaps in $Z_t^*$ of order at most $k_n+m_n-1$ are disjoint and contained in $Z_t$, and every such a gap is contained in some interval $I_s$ for $s\in \{0,1\}^{k_n+m_n-1}$. Moreover, distances between them (as well as between endpoints of $Z_t$ and the nearest to them gaps) are equal to numbers from the sequence $(M^{n+1}_i).$

{\bf Case 2.} $\delta_k$ satisfies $(**)$.
We will show that gaps are placed like on Figure \ref{rys**} (for $m_n=3$) which presents an overlap $Z_t$ and the distance between gaps $G$ and $H$ is one of the numbers $M^{n+1}_i$ (the other intervals between gaps or between endpoints of $Z_t$ and the nearest gaps are some intervals $I_v$ for $v\in \{0,1\}^{k_n+m_n-1}$). The letters correspond to the following gaps: $A=G_{t\ha 1\ha 0^{(k_n-k)}}$, $B=G_{t\ha 1\ha 0^{(k_n-k-1)}}$,
$C=G_{t\ha 1\ha 0^{(k_n-k)}\ha 1}$, $D=G_{t\ha 1\ha 0^{(k_n-k-2)}}$, $E=G_{t\ha 0\ha 1^{(k_n-k-2)}}$, $F=G_{t\ha 0\ha 1^{(k_n-k)}\ha 0}$, $G=G_{t\ha 0\ha 1^{(k_n-k-1)}}$, $H=G_{t\ha 0\ha 1^{(k_n-k)}}$.

 \begin{figure} [h] 
\includegraphics[width=0.9\textwidth]{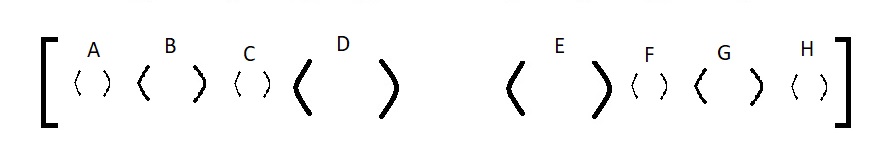}
\caption{}
\label{rys**}
\end{figure}

Analogously as in Case 1 we have 
$l(G_{t\ha 1 \ha 0^{(k_n-k-2)}}) > l(I_{t \ha 1}) = l(Z_t)$ and $r(G_{t\ha 0 \ha 1^{(k_n-k-2)}}) < r(I_{t \ha 0}) = r(Z_t)$.

We also have, by $(**)$,
$$l(G_{t\ha 0 \ha 1^{(k_n-k-2)}})-r(G_{t\ha 1 \ha 0^{(k_n-k-2)}}) = r(Z_t)-a_{k_n}-(l(Z_t)+a_{k_n}) = \delta_k -2a_{k_n}>0.$$
All other gaps in $Z_t^*$ of order at most $k_n+m_n-1$ are contained in $I_{t\ha 0 \ha 1^{(k_n-k-1)}}$ or $I_{t\ha 1 \ha 0^{(k_n-k-1)}}$, and, by the above calculations, these intervals are disjoint. Thus, all gaps in $Z_t^*$ of order at most $k_n+m_n-1$ are pairwise disjoint and contained in $Z_t$. Moreover, distance between gaps $G_{t\ha 1 \ha 0^{(k_n-k-2)}}$ and $G_{t\ha  0 \ha 1^{(k_n-k-2)}}$ is a number from the sequence $(M^{n+1}_i)$.

{\bf Case 3.}
$\delta_k$ satisfies $(*')$.
We will show that gaps are placed like on Figure \ref{rys*'} (for $m_n=4$) which presents an overlap $Z_t$ and distances between consecutive gaps (or endpoints of $Z_t$ and the nearest gaps) are equal to numbers from the sequence $(M^{n+1}_i)$. The letters correspond to the following gaps: $A=G_{t\ha 0\ha 1\ha 0\ha 0}$, $B=G_{t\ha 1\ha 0 \ha 0 \ha 0}$,
$C=G_{t\ha 0\ha 1\ha 0}$, $D=G_{t\ha 1\ha 0 \ha 0}$, $E=G_{t\ha 0\ha 1\ha 0\ha 1}$, $F=G_{t\ha 1\ha 0 \ha 0 \ha 1}$, $G=G_{t\ha 0\ha 1}$, $H=G_{t\ha 1\ha 0}$,
$J=G_{t\ha 0\ha 1\ha 1\ha 0}$, $K=G_{t\ha 1\ha 0\ha 1 \ha 0}$,
$L=G_{t\ha 0\ha 1\ha 1}$, $M=G_{t\ha 1\ha 0\ha 1}$,
$N=G_{t\ha 0\ha 1\ha 1\ha 1}$, $O=G_{t\ha 1\ha 0\ha 1 \ha 1}$.

 \begin{figure} [h] 
\includegraphics[width=0.9\textwidth]{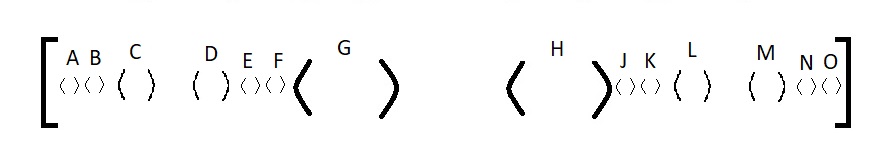}
\caption{}
\label{rys*'}
\end{figure}
We have
$G_{t\ha 1 \ha 0} \subset J_{t \ha 1}$, so $l(G_{t\ha 1 \ha 0}) > l(J_{t \ha 1}) = l(Z_t)$. Moreover, $$r(Z_t) - r(G_{t\ha 1 \ha 0}) = r(Z_t)-(l(Z_t) +a_{k_n+1})= |Z_t| - a_{k_n+1} = \delta_k - a_{k_n+1} > 0,$$ so
$G_{t\ha 1 \ha 0} \subset Z_t$. Similarly, $G_{t\ha 0 \ha 1} \subset Z_t$. By $(*')$, we also have
$$l(G_{t\ha 1 \ha 0}) - r(G_{t\ha 0 \ha 1}) = l(Z_t) + r_{k_n+1} - (r(Z_t)-r_{k_n+1}) =2r_{k_n+1} - |Z_t| = 2r_{k_n+1} - \delta_k >0.$$
So, $G_{t\ha 1 \ha 0} \cap G_{t\ha 0 \ha 1} = \emptyset$. In particular, $G_{t\ha 1 \ha 0} \subset I_{t \ha 0 \ha 1 \ha 1}$ and $G_{t\ha 0 \ha 1} \subset I_{t \ha 1 \ha 0\ha 0}.$

Assume that $G_{v}, G_{w} \in Z_t^*$, where $|v|=|w| = k_n+i$ for $0\leq i \leq m_n-3$ are such that $G_v<G_w$, there are no gaps from $Z_{t}^*$ of order at most $k_n+i+1$ in intervals $[l(I_w),l(G_v)], [r(G_v),l(G_w)], [r(G_w),r(I_v)]$, $l(G_v)-l(I_w)=r(I_v)-r(G_w) = \delta_{k} - \sum_{j=1}^{i+1}a_{k_{n}+j}$ and
$l(G_w)-r(G_v) = 2r_{k_n+i+1} - \delta_k+\sum_{l=1}^i a_{k_n+l}$.
 
Using $(*')$, we get
$$l(G_{v\ha 0})-l(I_w) = l(G_v)-a_{k_n+i+2}-l(I_w) = \delta_{k} - \sum_{j=1}^{i+1}a_{k_{n}+j} -a_{k_n+i+2} = \delta_{k} - \sum_{j=1}^{i+2}a_{k_{n}+j}\geq \delta_{k} - \sum_{j=1}^{m_n-1}a_{k_{n}+j} >0.$$
Analogously, 
$$l(G_v)-r(G_{w\ha 0}) = l(G_{v\ha 1}) - r(G_w) = r(I_v) - r(G_{w\ha 1}) = \delta_{k} - \sum_{j=1}^{i+2}a_{k_{n}+j} > 0.$$
Moreover, 
$$l(G_{w\ha 0})-r(G_{v\ha 0}) = l(I_w)+r_{k_n+i+2} - (l(G_v)-r_{k_n+i+2}) =  -\delta_{k} + \sum_{j=1}^{i+1}a_{k_{n}+j}+2r_{k_n+i+2}.$$

Similarly, $$l(G_{w\ha 1})-r(G_{v\ha 1}) =-\delta_{k} + \sum_{j=1}^{i+1}a_{k_{n}+j}+2r_{k_n+i+2}>0.$$ 

Finally, we get that $l(I_w)<G_{v\ha 0}<G_{w\ha 0} < G_v$ and $G_w<G_{v\ha 1}<G_{w\ha 1} < r(I_v)$ and there are no gaps between them in $Z_t^*$ of order at most $k_n+i+2$. In particular, $G_{v\ha 0} \subset I_{w\ha 0 \ha 0}$, $G_{w\ha 0} \subset I_{v\ha 0 \ha 1}$, $G_{v\ha 1} \subset I_{w\ha 1 \ha 0}$ and $G_{w\ha 1} \subset I_{v\ha 1 \ha 1}$. By induction, all gaps in $Z_t^*$ of order at most $k_n+m_n-1$ are pairwise disjoint and contained in $Z_t$ and every such a gap is contained in some interval $I_s$ for $s \in \{0,1\}^{k_n+m_n-1}$. Moreover, distances between them (as well as between endpoints of $Z_t$ and the nearest to them gaps) are equal to numbers from the sequence $(M^{n+1}_i).$

Now, assume that for some $j \in \N$ all gaps in $Z_t^*$ of order at most $k_j+m_j-1$ are pairwise disjoint and contained in $Z_t$ and distances between them (as well as between endpoints of $Z_t$ and the nearest to them gaps) are equal to numbers from the sequence $(M^{j+1}_i).$ Since all numbers $M_i^{j+1}$ satisfy either $(*)$ or $(**)$, then we can repeat the reasoning from Case 1 or Case 2 to show that all gaps in $Z_t^*$ of order at most $k_{j+1}+m_{j+1}-1$ are pairwise disjoint and contained in $Z_t$. By induction, all gaps in $Z_t^*$ are pairwise disjoint and contained in $Z_t$.

Now, we are going to inductively show that every gap $G$ either belongs to $\G$ or there is $t$ such that $G \in Z_t^*$. 

Let $s \in \{0,1\}^{k_1-1}$. If $s_i = 0$ or $s_i = 1$ for all $i$, then $G_s$ comes from $0$ or $1$, so $G_s \in \G$. If this is not the case, then there exists $N <k_1-1$ such that $s_N \neq s_{i}$ for $i > N$. Then $G_s$ comes from $l(I_{s|(N-1) \ha 1}) = l(Z_{s|(N-1)})$ or from $r(I_{s|(N-1) \ha 0}) = r(Z_{s|(N-1)})$, so $G_s \in Z_{s|(N-1)}^*$.

Now, if $m_1>1$, let $s\in \{0,1\}^{k_1+i}$ for $i\in \{0,\dots,m_1-2\}$. If $s_{k_1} =s_{k_1-1}$, then $G_s$ is a neighbour of $G_{s|k_1-1}$. We have already showed that $G_{s|k_1-1} \in \G$ or $G_{s|k_1-1} \in Z_t^*$ for some $t$. Hence if $G_{s|k_1-1} \in \G$, then also $G_s \in \G$ and if $G_{s|k_1-1} \in Z_t^*$, then also $G_{s} \in Z_t^*$. If $s_{k_1} \neq s_{k_1-1}$, then, by the definition of an overlap of the second type, we have $G_s \in Z_{s|(k_1-2)}^*$.

Now, assume that all gaps of order at most $k_n+m_n-1$ are either in $\G$ or in some $Z_t^*$. 
Let $s \in \{0,1\}^{k_{n+1}-1}$. If $s_i = 0$ or $s_i = 1$ for all $i$, then $G_s$ comes from $0$ or $1$, so $G_s \in \G$. If this is not the case, then there exists $N <k_{n+1}-1$ such that $s_N \neq S_{i}$ for $i > N$. If $N \notin K$, then $G_s$ comes from $l(Z_{s|(N-1)})$ or $r(Z_{s|(N-1)})$, so $G_s \in Z_{s|(N-1)}^*$. If $N \in K$, then $G_s$ is a gap coming from $G_{s|(N-1)}$ which is a gap of order at most $k_n+m_n-1$. If $G_{s|(N-1)} \in \G$, then also $G_s \in \G$ and if $G_{s|(N-1)} \in Z_t^*$ for some $t$, then also $G_s \in Z_t^*$. 

Now, if $m_{n+1}>1$, let $s\in \{0,1\}^{k_{n+1}+i}$ for $i\in \{0,\dots,m_{n+1}-2\}$. If $s_{k_{n+1}} =s_{k_{n+1}-1}$, then $G_s$ is a neighbour of $G_{s|k_{n+1}-1}$. We have already showed that $G_{s|k_{n+1}-1} \in \G$ or $G_{s|k_{n+1}-1} \in Z_t^*$ for some $t$. Hence if $G_{s|k_{n+1}-1} \in \G$, then also $G_s \in \G$ and if $G_{s|k_{n+1}-1} \in Z_t^*$, then also $G_{s} \in Z_t^*$. If $s_{k_{n+1}} \neq s_{k_{n+1}-1}$, then, by the definition of an overlap of the second type, we have $G_s \in Z_{s|(k_{n+1}-2)}^*$.

By induction, we have that every gap $G_s$, where $|s|+1\in K$, belongs either to $\G$ or to $Z_t^*$ for some $t$.

We will now inductively show that for every $x \in I\setminus \bigcup \G$ and every $n \in \N$ there exists $t\in \{0,1\}^{k_n+m_n-1}$ such that $x\in I_t.$
Let $x \in I\setminus \bigcup \G$. Assume that $x \notin I_p$ for any $p \in\{0,1\}^{k_1+m_1-1}$. Then $x \in G_s$ for some $s\in\{0,1\}^{<k_1+m_1-1}$. Since $x \notin \bigcup \G$, we have $G_s \notin \G$, and so, $G_s \in Z_t^*$ for some $t$. By the assumption, $\delta_{|t|+1}$ satisfies $(*)$, $(**)$ or $(*')$. If $\delta_{|t|+1}$ satisfies $(*)$ or $(*')$, then we showed that $G_s \subset I_w$ for some $w \in \{0,1\}^{k_1+m_1-1}$, a contradiction. Hence assume that $\delta_{|t|+1}$ satisfies $(**)$. Since $G_s \subset Z_t$, we have $G_s \subset I_{t\ha (1-j)}$, where $j = s_{|t|+1}$. Since $x \notin I_p$ for any $p \in\{0,1\}^{k_1+m_1-1}$, $x \in G_w$, where $ t\ha (1-j) \prec w$. We know that $G_s$ is disjoint with $G_{t\ha (1-j)\ha j^{(k_1-2-|t|)}}$. Without loss of generality suppose that $j=1$. We have $r(G_{t\ha 0^{(k_1-1-|t|)}})=l(I_t)+a_{k_1} <l(I_t)+a_{|t|+1}\leq l(Z_t),$ so $G_{t\ha 0^{(k_1-1-|t|)}} \cap G_s = \emptyset$. Since we cannot have $G_w \in \G$, there is $v$ such that $G_w \in Z_v^*$. We also have $G_w \neq G_{t\ha 0^{(k_1-1-|t|)}}$ and $G_w \neq G_{t\ha 0 \ha 1^{(k_1-2-|t|)}}$ (and it is also disjoint with neighbours of these gaps). Thus, $t \ha 0 \prec v$, and so $|v| > |t|$. 
If $|Z_v|$ satisfies $(*)$ or $(*')$, then we obtain a contradiction, similarly as earlier. Otherwise, we repeat the reasoning as above to show that $x$ is in a gap which belongs to $Z_u^*$ with $|u| > |v|$. Repeating this reasoning finitely many times we will obtain that $x$ is in a gap which belongs to $Z_q^*$ which satisfies $(*)$ or $(*')$ (because eventually $|q| = k_1-2$, and so $|Z_q| = \delta_{k_1-1}$ which satisfies $(*')$). Thus, $x \in I_p$ for some $p \in\{0,1\}^{k_1+m_1-1}$.

Now, assume that for some $n \in \N$ and some $p \in \{0,1\}^{k_n+m_n-1}$ we have $x \in I_p$. Suppose that $x \notin I_q$ for any $q \in \{0,1\}^{k_{n+1}+m_{n+1}-1}$. Then $x \in G_s$ for some $s \in\{0,1\}^{<k_{n+1}+m_{n+1}-1}$ with $p \prec s$. 
If $G_s$ is not a gap coming from $l(I_p)$ or $r(I_p)$, or one of their neighbours, then there exists $w$ with $|w| \geq k_{n}+m_n-1$ such that $G_s \in Z_w^*$. Repeating the reasoning as above, we obtain that $x \in I_q$ for some $q \in \{0,1\}^{k_{n+1}+m_{n+1}-1}$.
Now, without loss of generality assume that $G_s$ is the gap coming from $l(I_p)$ (or neighbour of this gap). Since $G_s \notin \G$, we have $l(I_p) \neq 0$. Thus, $N=\max\{i\colon p_i =1\}>0$. If $N \in K$, then $l(I_p)$ is an endpoint of some gap $G_w$ belonging to some $Z_v^*$. If $N \notin K$, then $l(I_p) = l(Z_v)$, where $v = p|(N-1)$. In both cases $G_s \in Z_v^*$. Since all gaps in $Z_v^*$ are disjoint, we have that $G_s$ lies in one of the intervals $J$ between gaps from $Z_v^*$ of order at most $k_n+m_n-1$ (or between endpoint of $Z_v$ and a gap in $Z_v^*$). We have already showed that either $|J|= M^{n+1}_i$ (in Case 1, Case 3 and sometimes in Case 2) or $J = I_p$ (in Case 2 if the gaps between which lies $J$ are neighbours of the same gap). If $|J|= M^{n+1}_i$, then this number satisfies $(*)$ or $(**)$ and, reasoning similarly as earlier, we obtain that $x \in I_q$ for some $q \in \{0,1\}^{k_{n+1}+m_{n+1}-1}$. 
If $J=I_p$, the interval (overlap) in which $G_w$ was contained had length satisfying $(**)$ (or $|Z_v^*|$ satisfied $(**)$). But again, repeating the reasoning from above, we obtain that $x \in I_q$ for some $q \in \{0,1\}^{k_{n+1}+m_{n+1}-1}$.

Thus, for any $n \in \N$ we have $I \setminus \bigcup\G \subset \bigcup_{t\in\{0,1\}^{k_n+m_n-1}}$, and so
$I \setminus \bigcup\G \subset E(a_n)$. Therefore, $E(a_n)$ is a Cantorval. 

%

\end{proof}
\section{Examples and applications}

In this section we would like to show how we can use Theorem \ref{main} to produce new Cantorvals.
One class of sequences which is specially interesting in this context is the class of multigeometric sequences.
For any $m\in \N$, $q\in (0,1)$, $k_1,k_2, \dots,k_m \in \R$. The sequence $(k_1q, k_2q,\dots,k_mq,k_1q^2,k_2q^2,\dots, k_mq^2,\dots)$ is called multigeometric and is denoted by $(k_1,k_2,\dots,k_m;q)$.
Let us begin with the following example.

\begin{example}
Consider the set $E(1,2\sqrt{2}-2;q),$ where $q=\frac{2-\sqrt{2}}{2}$. We will show that this set is a Cantorval. 
First, observe that $k_n = 2n$. Indeed, 
$$r_{2}-{a_2} = q((2\sqrt{2}-1)\cdot\frac{q}{1-q} -2\sqrt{2}+2) \approx -0,07q <0,$$ 
so $r_2 <a_2$.
Moreover, $r_{2n} = r_2\cdot q^{n-1}$ and $a_{2n} = a_2 \cdot q^{n-1}$, so $r_{2n}<a_{2n}$ for $n\in \N$. On the other hand,
$r_1-a_1 = q((2\sqrt{2}-1)\cdot\frac{q}{1-q} +2\sqrt{2}-2-1) \approx
 0,59q>0$, so $r_1 > a_1$. Moreover, $r_{2n-1}=r_1\cdot q^{n-1}$ and $a_{2n-1}=a_1\cdot q^{n-1}$, so $r_{2n-1}> a_{2n-1}$ for $n\in \N$.

We will show that the assumptions of Theorem \ref{main} are satisfied.
Let $n \in \N$. If $n=1$, then there is no $i \notin K$, $i<k_1-1=1$, so there is no $M^1_1$ and $\delta_{1} > 0$, so $(*')$ is satisfied.
If $n > 1$, then 
$$\delta_{2n-3}-2a_{2n}= r_{2n-3}-a_{2n-3}-2a_{2n} = r_1\cdot q^{n-2} - a_1q^{n-2}-2a_2q\cdot q^{n-2} = q^{n-2} \cdot (r_1-a_1-2a_2q) \approx 0,1q^{n-1} >0,$$
so it satisfies $(**)$.

Now, if $n > 2$, then for
$M= \delta_{2n-5}-2a_{2n-2}$ 
we have $$M-a_{2n} = q^{n-3} (r_1-a_1-2a_2q-a_2q^2) \approx 0,03q^{n-2} > 0.$$
After calculations we also get
$$2r_{2n}-M = M-a_{2n} > 0.$$
So, the condition $(*)$ is satisfied. If $n >3$, then
after easy but tedious calculations it occurs that $\delta_{2n-7}-2a_{2n-4} -a_{2n-2} = 2r_{2n-2} - (\delta_{2n-7}-2a_{2n-4}) = \delta_{2n-5}-2a_{2n-2}=M$, so it also satisfies $(*)$. As the consequence we get that all remaining numbers $M^n_i$ are equal to $M$, so condition $(*)$ holds for all of them. Therefore, the assumptions of Theorem \ref{main} are satisfied, so $E(a_n)$ is a Cantorval.

\end{example}
It is worth noting that this is probably the first example of multigeometric sequence with non-commensurable coefficients which achievement set is a Cantorval.

As it was mentioned earlier, in the paper \cite{MM} the authors proved that for any infinite $L\subset \N$ there exists a nonincreasing sequence $(a_n)$ such that $E(a_n)$ is a Cantor set with $K=L$.
They asked if the same can be proved for Cantorvals.
Now, we answer this question.

\begin{theorem}\label{mami}

Let $(k_n)$, $(m_n)$ be sequences of natural numbers such that $k_1>1$ and $k_{n+1} > k_n+m_n$. Put $k_0:=0$, $m_0:=1,$ $K_n := \{k_n,\dots,k_n+m_n-1\}$, $K=\bigcup_{n=1}^\infty K_n$.
Let $a_1 > 0$ and for $i > 1$ 
$$a_i := \left\{ \begin{array}{ccc}
\frac12 a_{i-1} \;\text{ if }\; i-1,i,i+1 \notin K \text{ or }  i-1,i \in K\\ 
\frac{2^{m_n}+1}{2^{m_n+1}} a_{i-1} \;\text{ if }\; i=k_n-1> k_{n-1}+m_{n-1}\\ 
\frac{2^{m_n}+1}{2^{m_n+2}} a_{i-1} \;\text{ if }\;  i=k_n-1=k_{n-1}+m_{n-1}\\
\frac{2^{m_n}}{2^{m_n}+1} a_{i-1} \;\text{ if }\;i=k_n \\ 
\frac14 a_{i-1} \;\text{ if }\; i,i+1 \notin K, i-1\in K.
\end{array}%
\right. $$
Then the sequence $(a_n)$ is such that $a_n > r_n$ if and only if $n \in K$ and $E(a_n)$ is a Cantorval. 
\end{theorem}
\begin{proof}
Let $n\in \N$ and $j \in K_n$. We will show that $r_j<a_j$. 
Observe that if $a_i,a_{i+1} \in K_n$ and $r_{i+1} < a_{i+1},$ then also $r_i < a_i$. Indeed, we have $a_{i+1} =\frac12a_i$, so $r_i = r_{i+1}+a_{i+1}<a_{i+1}+a_{i+1}=2a_{i+1} = a_i.$ Thus, it suffices to show that $a_{j} > r_{j}$ for $j=k_n+m_n-1$.

If $k_n+m_n < k_{n+1}-1$, then $$a_{k_n+m_n} = \frac14 a_{k_n+m_n-1},$$ $$a_{k_{n+1}-1} = \frac{2^{m_{n+1}}+1}{2^{m_{n+1}+1}}a_{k_{n+1}-2} \leq \frac34 \cdot \frac14 a_{k_n+m_n-1} < \frac12 a_{k_n+m_n-1},$$ $$a_{k_{n+1}} = \frac{2^{m_{n+1}}}{2^{m_{n+1}}+1}a_{k_{n+1}-1} = \frac12a_{k_{n+1}-2}$$ and $a_i = \frac12 a_{i-1}$ for the remaining $i \in \{k_n+m_n,\dots,k_{n+1}+m_{n+1}-1\}$. 
If $k_n+m_n = k_{n+1}-1$, then $$a_{k_n+m_n} = \frac{2^{m_{n+1}}+1}{2^{m_{n+1}+2}}a_{k_n+m_n-1} < \frac12 a_{k_n+m_n-1},$$ $$a_{k_{n+1}} = \frac{2^{m_{n+1}}}{2^{m_{n+1}}+1}a_{k_{n+1}-1}=\frac14a_{k_{n+1}-2}=\frac14a_{k_{n}+m_n-1}$$ and 
$a_i = \frac12 a_{i-1}$ for the remaining $i \in \{k_n+m_n,\dots,k_{n+1}+m_{n+1}-1\}$. 

Therefore,
$$r_{k_n+m_n-1} = \sum_{i=k_{n}+m_n}^\infty a_i < \sum_{i=1}^\infty \frac1{2^i} a_{k_n+m_n-1} = a_{k_n+m_n-1}.$$

Now we will show that for every $n \in \N$, $\delta_{k_n-1}-\sum_{i=1}^{m_n-1} a_{k_n+i} >0$, and so $\delta_{k_n-1} >0$ which is equivalent to $r_{k_n-1} > a_{k_n-1}$. Later we will show that all numbers $M^{n+1}_i$ are equal to $\delta_{k_n-1}-\sum_{i=1}^{m_n-1} a_{k_n+i}$, and so are positive (and so are $\delta_i$ for $i \notin K$, which will prove that $a_i > r_i$ if and only if $i \in K$). 

We have $a_{k_n}=\frac{2^{m_n}}{2^{m_n}+1} a_{k_n-1}$ and
\begin{equation}\label{kn-1}
\begin{aligned}
\delta_{k_n-1}- \sum_{i=1}^{m_n-1} a_{k_n+i} &=r_{k_n-1}-a_{k_n-1}-\sum_{i=1}^{m_n-1} a_{k_n+i}  = a_{k_n}-a_{k_n-1}+ r_{k_n+m_n-1} \\&= a_{k_n} - \frac{2^{m_n}+1}{2^{m_n}}a_{k_n}+r_{k_n+m_n-1} = -\frac1{2^{m_n}}a_{k_n} + r_{k_n+m_n-1} = -\frac12a_{k_n+m_n-1}+r_{k_n+m_n-1}.
\end{aligned}
\end{equation}


Consider the cases.

1. $k_n+m_n+1 = k_{n+1}$
Then we have $a_{k_n+m_n} = \frac{2^{m_n}+1}{2^{m_n+2}} a_{k_n+m_n-1}$ and $a_{k_n+m_n+1} = \frac{2^{m_n}}{2^{m_n}+1} a_{k_n+m_n},$ so 
$$-\frac12a_{k_n+m_n-1}+r_{k_n+m_n-1} > -\frac12\cdot \frac{2^{m_n+2}}{2^{m_n}+1} a_{k_n+m_n}+a_{k_n+m_n}+a_{k_n+m_n+1}$$$$=-\frac{2^{m_n+1}}{2^{m_n}+1} a_{k_n+m_n}+a_{k_n+m_n}+\frac{2^{m_n}}{2^{m_n}+1} a_{k_n+m_n}=\frac{a_{k_n+m_n}}{2^{m_n+1}}(-2^{m_n+1}+2^{m_n}+2^{m_n}+1)=\frac{a_{k_n+m_n}}{2^{m_n+1}}>0.$$
 
2. $k_{n+1} = k_{n}+m_n + j$, where $j\geq 2$.
Then we have $a_{k_n+m_n} = \frac14 a_{k_n+m_n-1}$, $a_{k_n+m_n+i} = \frac12 a_{k_n+m_n+i-1}$ for $1\leq i\leq j-2$, $$a_{k_n+m_n+j-1} = \frac{2^{m_n}+1}{2^{m_n+1}}a_{k_n+m_n+j-2}> \frac12 a_{k_n+m_n+j-2},$$ $$a_{k_n+m_n+j} =\frac{2^{m_n}}{2^{m_n}+1}a_{k_n+m_n+j-1}=\frac12 a_{k_n+m_n+j-2},$$ so
$$-\frac12a_{k_n+m_n-1}+r_{k_n+m_n-1} > -2a_{k_n+m_n} +a_{k_n+m_n} + a_{k_n+m_n+1} + \dots +a_{k_n+m_n+j}$$$$ > a_{k_n+m_n}(-1 +\sum_{i=1}^{j-2} \frac1{2^i} + \frac1{2^{j-1}}+\frac1{2^{j-1}}) = 0.$$

We will now show that for any $n \in \N$ numbers $M^n_i$ are equal for all $i$ (and for $n>1$ are equal to $ \delta_{k_{n-1}-1}-\sum_{i=1}^{m_{n-1}-1} a_{k_n+i}$) and satisfy condition $(*)$ and $\delta_{k_n-1}$ satisfies $(*')$. In particular, $\delta_i>0$ for $i \notin K$, which will prove that $a_i > r_i$ if and only if $i \in K$.

If $k_1=2$, then there is no number $M^1_i$ defined. If $k_1=3$, then we have only one such number. If $k_1\geq 4$, then $M^1_i = \delta_i$ for $i < k_1-1$. For $i\in \{1,2,\dots,k_1-3\}$ we have $a_{i+1} = \frac{1}{2}a_i,$ and so 
\begin{equation} \label{delta}
\delta_{i+1} = r_{i+1}-a_{i+1} = r_i-2a_{i+1}=r_i - 2\cdot \frac{1}{2}a_i = r_i - a_i = \delta_i.
\end{equation}
We have $a_{k_1-1} = \frac{2^{m_1}+1}{2^{m_1+1}} a_{k_1-2}$, $a_{k_1} = \frac{2^{m_1}}{2^{m_1}+1} a_{k_1-1}$ and $a_{k_1+i}=\frac12 a_{k_1+i-1}$ for $1\leq 1 <k_1+m_1$, and so $a_{k_1-2}=2^{m_1}a_{k_1+m_1-1}$. Therefore, by (\ref{kn-1}),
\begin{equation}\label{delta-a}
\begin{aligned}
\delta_{k_1-2} - a_{k_1} &= r_{k_1-2}-a_{k_1-2}-a_{k_1} = -a_{k_1-2}-a_{k_1}+r_{k_1+m_1-1}+a_{k_1-1}+a_{k_1} +\dots + a_{k_1+m_1-1}\\&=r_{k_1+m_1-1}+a_{k_1+m_1-1} (1+2+\dots+2^{m_1-2}+2^{m_1-1}\cdot\frac{2^{m_1}+1}{2^{m_1}} -2^{m_1})\\&= r_{k_1+m_1-1}-\frac12 a_{k_1+m_1-1} = \delta_{k_{1}-1}-\sum_{i=1}^{m_{1}-1} a_{k_n+i} > 0,
\end{aligned}
\end{equation} 
\begin{equation}\label{2r-delta}
\begin{aligned}
r_{k_1}+r_{k_1+m_1-1} - \delta_{k_1-2}&=r_{k_1}+r_{k_1+m_1-1}-r_{k_1-2}+a_{k_1-2}=r_{k_1+m_1-1}-a_{k_1-1} -a_{k_1}+a_{k_1-2} \\&=r_{k_1+m_1-1}-\frac{2^{m_1}+1}{2^{m_1}}a_{k_1} -a_{k_1}+2a_{k_1}=r_{k_1+m_1-1} - \frac12a_{k_1+m_1-1} \\&=\delta_{k_{1}-1}-\sum_{i=1}^{m_{1}-1} a_{k_n+i} >0
\end{aligned}
\end{equation}
and for $j \in \{1, \dots, m_1-1\}$ (if $m_1>1$) (by (\ref{delta-a}))
\begin{equation}\label{M-2-a}
\begin{aligned}
\delta_{k_1-2} - 2a_{k_1+j}-\sum_{l=1}^{j-1}a_{k_1+l} &= r_{k_1-2} - a_{k_1-2} -a_{k_1+j-1} -\sum_{l=1}^{j-1}a_{k_1+l} =r_{k_1-2} - a_{k_1-2} - \sum_{l=1}^{j-1}a_{k_1}\frac1{2^l} - a_{k_1}\frac1{2^{j-1}}\\&=r_{k_1-2} - a_{k_1-2}-a_{k_1} =        \delta_{k_{1}-1}-\sum_{i=1}^{m_{1}-1} a_{k_1+i} >0.
\end{aligned}
\end{equation}
Thus, $\delta_i$ are equal to $\delta_{k_{1-1}-1}-\sum_{i=1}^{m_{n-1}-1} a_{k_n+i}$ and satisfy $(*)$ for $i < k_1-1$.

We will now show that $\delta_{k_1-1}$ satisfy $(*')$.
Let $j \in \{1,\dots, m_1-1\}$. 
We have, by (\ref{delta-a}),
\begin{equation}\label{2r*'}
\begin{aligned}
2r_{k_1+j}-\delta_{k_1-1} + \sum_{i=1}^{j-1}a_{k_1+i} &= 2r_{k_1+j}-r_{k_1-1}+a_{k_1-1} + \sum_{i=1}^{j-1}a_{k_1+i} = 2r_{k_1+j}-r_{k_1+j-1}-a_{k_1}+a_{k_1-1}\\&=r_{k_1+j}-a_{k_1+j}-a_{k_1}+a_{k_1-1}\\&=r_{k_1-2}-a_{k_1}-a_{k_1+j}-a_{k_1}-a_{k_1+1}-\dots-a_{k_1+j} \\&=r_{k_1-2}-a_{k_1} - a_{k_1-2} (\frac12 + \frac14 + \dots + \frac1{2^{j-2}}+\frac{2}{2^{j-1}})\\& = r_{k_1-2}-a_{k_1-2}-a_{k_1} = \delta_{k_1-1}-\sum_{i=1}^{m_1-1}a_{k_1+i} > 0,
\end{aligned}
\end{equation}
and hence $\delta_{k_1-1}$ satisfy $(*')$.

Let $n\in \N$. Assume that numbers $M^{n}_i$ are equal to some $M >0$ for all $i$, $M$ satisfies $(*)$, $\delta_{k_n-1}$ satisfies $(*')$ and all left-hand sides of inequalities $(*)$ and $(*')$ are equal to $-\frac12a_{k_n+m_n-1}+r_{k_n+m_n-1}$. We will show that all numbers $M^{n+1}_i$ are equal and satisfy $(*)$, $\delta_{k_{n+1}-1}$ satisfies $(*')$ and all left-hand sides of inequalities $(*)$ and $(*')$ are equal to $\delta_{k_{n+1}-1}-\sum_{j=1}^{m_{n+1}-1}a_{k_n+j} = -\frac12a_{k_{n+1}+m_{n+1}-1}+r_{k_{n+1}+m_{n+1}-1}.$

Consider the cases.

1. $k_n+m_n < k_{n+1}-1$.
Then $a_{k_n+m_n} = \frac14 a_{k_n+m_n-1}$.
So, 
$$\delta_{k_n+m_n} = r_{k_n+m_n} - a_{k_n+m_n} = r_{k_n+m_n-1} - a_{k_n+m_n} - a_{k_n+m_n} = r_{k_n+m_n-1} - \frac12 a_{k_n+m_n-1}.$$
Repeating the reasoning from $(\delta)$, we obtain that all numbers $\delta_i$ for $k_n+m_n \leq i <k_{n+1}-1$ are equal, and so all numbers $M^{n+1}_i$ are equal. Hence we can repeat the reasoning from the proof for $n=1$ to show that they satisfy $(*)$, $\delta_{k_{n+1}-1}$ satisfies $(*')$ and all left-hand sides of inequalities $(*)$ and $(*')$ are equal to $\delta_{k_{n+1}-1}-\sum_{j=1}^{m_{n+1}-1}a_{k_n+j} = -\frac12a_{k_{n+1}+m_{n+1}-1}+r_{k_{n+1}+m_{n+1}-1}.$

2. $k_n+m_n = k_{n+1}-1$.
Then the only numbers $M^{n+1}_i$ come from the inequalities $(*)$ and $(*')$, and by the assumption, they are equal to $r_{k_n+m_n-1} - \frac12 a_{k_n+m_n-1}$. 
We already know that if $k_n+m_n < k_{n+1}-1$, then all differences between left-hand sides of conditions $(*)$ or $(*')$ and $\delta_{k_{n+1}-1}-\sum_{j=1}^{m_{n+1}-1}a_{k_n+j}$ are equal to $0$, but these differences depend only on numbers $M^{n+1}_i$, $\delta_{k_{n+1}-1}$ and $a_i$ for $i\geq k_{n+1}-1$. However, the definition of $a_i$ for $i \geq k_{n+1}$ are the same with respect to $a_{k_{n+1}-1}$ regardless if $k_n+m_n = k_{n+1}-1$ or  $k_n+m_n < k_{n+1}-1$. Of course, also $\delta_{k_{n+1}-1}$ depends only on $a_i$ for $i \geq k_{n+1}-1$. Hence to show that also if $k_n+m_n = k_{n+1}-1$, then all differences between left-hand sides of conditions $(*)$ or $(*')$ and $\delta_{k_{n+1}-1}-\sum_{j=1}^{m_{n+1}-1}a_{k_n+j}$ are equal to $0$ it suffices to show that $M^{n+1}_i$ expressed in terms of $a_i$ for $i \geq k_n+m_n-1$ is the same regardless $k_n+m_n = k_{n+1}-1$ or $k_n+m_n < k_{n+1}-1$. If
$k_n+m_n=k_{n+1}-1$, then $a_{k_{n+1}-1} = \frac{2^{m_{n+1}}+1}{2^{m_{n+1}+2}} a_{k_{n+1}-2}$, and so
$$M^{n+1}_i  =r_{k_n+m_n-1}-\frac12 a_{k_n+m_n-1} = r_{k_{n+1}-2}-\frac12 a_{k_{n+1}-2} = r_{k_{n+1}-2} - \frac12 \cdot \frac{2^{m_{n+1}+2}}{2^{m_{n+1}}+1}a_{k_{n+1}-1} $$$$= r_{k_{n+1}-2} - \frac{2^{m_{n+1}+1}}{2^{m_{n+1}}+1}a_{k_{n+1}-1}.$$
If
$k_n+m_n<k_{n+1}-1$, then $a_{k_{n+1}-1} = \frac{2^{m_{n+1}}+1}{2^{m_{n+1}+1}} a_{k_{n+1}-2}$, and so
$$M^{n+1}_i  =\delta_{k_{n+1}-2}= r_{k_{n+1}-2}-a_{k_{n+1}-2} = r_{k_{n+1}-2} - \frac{2^{m_{n+1}+1}}{2^{m_{n+1}}+1}a_{k_{n+1}-1}.$$
Hence $M^{n+1}_i$ are the same (written using $a_i$ for $i \geq k_{n+1}-1$) in both cases. Therefore, by case 1, we obtain that numbers $M^{n+1}_i$ satisfy $(*)$, $\delta_{k_{n+1}-1}$ satisfies $(*')$ and all left-hand sides of inequalities $(*)$ and $(*')$ are equal to $\delta_{k_{n+1}-1}-\sum_{j=1}^{m_{n+1}-1}a_{k_n+j} = -\frac12a_{k_{n+1}+m_{n+1}-1}+r_{k_{n+1}+m_{n+1}-1}.$

By induction, the assumptions of Theorem \ref{main} are satisfied, so $E(a_n)$ is a Cantorval.
Moreover, $a_n >r_n$ if and only if $n\in K$.

\end{proof}
\begin{example}
Let $k_n=2n$ and $m_n=1$ for all $n \in \N$. We will use Theorem \ref{mami}. For $n\in \N$ we have  $a_{2n}=\frac23 a_{2n-1}$ and $a_{2n+1} = \frac38 a_{2n}.$ In particular, if we put $a_1 = \frac34$, we get $a_2 = \frac12$, $a_{2n} = a_2 \cdot (\frac 23\cdot \frac38)^{n-1} = \frac12 \cdot \frac1{4^{n-1}}$ and $a_{2n+1} = a_1 \cdot (\frac 23\cdot \frac38)^{n} = \frac34 \cdot \frac1{4^{n}}$, and thus $E(a_n)$ is the well-known Guthrie-Nymann Cantorval $E(3,2;\frac14)$.

\end{example}

In the papers \cite{BFS}, \cite{BBFS} the authors considered some special multigeometric sequences $x=(x_1,x_2,\dots,x_m;q)$. They proved that for some $q_1<q_2<q_3$ we have that if $q\geq q_3$, then $E(x)$ is an interval, if $q\in [q_2,q_3)$, then $E(x)$ is a Cantorval and if $q<q_1$, then $E(x)$ is a Cantor set. Moreover, under some additional assumption if $q=q_1$, then $E(x)$ is also a Cantorval. It was also proved that almost all $q \in (q_1,q_2)$ (in the sense of measure) $E(x)$ has a positive Lebesgue measure, but also there is a decreasing sequence $p_n$ in this intervals such that $E(x)$ is a Cantor set of Lebesgue measure zero for $q =p_n$. Nevertheless, in general it is not known what is the topological form of the set $E(x)$ for $q \in (q_1,q_2)$. In particular, there were no examples of $q$ in this interval for which $E(x)$ is a Cantorval. We will show how Theorem \ref{main} corresponds to the above results on the example of $E(4,3,2;q)$. For this example we have $q_1=\frac18, q_2= \frac16, q_3=\frac2{11}$. More on such sets was also showed in \cite{recover} and \cite{GM23}. Results concerning multigeometric sequences with ratio $q_1$ can be also found in \cite{Ke}, \cite{Ni} and \cite{GK}.

\begin{example} \label{18}
Let $k_n=3n$ and $m_n=1$ for all $n\in \N$. We will again use Theorem \ref{mami}. For $n\in \N$ we have  $a_{3n}=\frac23 a_{3n-1}$, $a_{3n-1} = \frac34 a_{3n-2}$ and $a_{3n+1} = \frac14a_{3n}.$ In particular, if we put $a_1 = \frac1{2}$, we get $a_2 = \frac38$, $a_3 = \frac14$, $a_{3n} = a_3 \cdot (\frac23\cdot \frac34 \cdot \frac14 )^{n-1} = \frac14 \cdot \frac1{8^{n-1}}$ and $a_{3n-1} = a_2 \cdot (\frac18)^{n-1} = \frac38 \cdot \frac1{8^{n-1}}$ and $a_{3n+1} = a_1\cdot \frac1{8^n}$, and thus $E(a_n)$ is the known Cantorval $E(4,3,2;\frac18)$.
\end{example}

\begin{example}\label{przedz}
It is known that $E(4,3,2;q)$ is a Cantorval for $q\in [\frac16,\frac2{11})$. We will show that it also follows from Theorem \ref{main}. First, observe that $k_n=3n$ and $m_n=1$. 

Indeed, for $n\in \N$ we have 
$$r_{3n}-a_{3n}=q^{n-1}(r_3-a_3) = q^n(\frac{9q}{1-q}-2)=q^n(\frac{9q-2+2q}{1-q})=q^n(11q-2) <0;$$
$$\delta_{3n-1}=r_{3n-1}-a_{3n-1}=q^{n-1}(r_2-a_2)=q^{n}(\frac{9q}{1-q}+2-3)=q^{n}(\frac{9q-1+q}{1-q})> 0,$$
$$\delta_{3n-2}=r_{3n-2}-a_{3n-2}= q^{n-1}(r_1-a_1) = q^{n} (\frac{9q}{1-q}+5-4)=q^{n} (\frac{9q+1-q}{1-q})>0.$$

We have 
$$\delta_{3n-2} - a_{3n} = q^{n} (\frac{8q+1}{1-q}-2)=q^{n} (\frac{10q-1}{1-q}) = \delta_{3n-1}>0$$
and
$$2r_{3n} - \delta_{3n-2} = 2q^n\frac{9q}{1-q} - q^n \frac{8q+1}{1-q} = q^n \frac{18q-8q-1}{1-q} = q^n\frac {10q-1}{1-q} = \delta_{3n-1}>0.$$
So, $\delta_{3n-2}$ satisfies $(*)$ for all $n \in \N$. Of course, $\delta_{3n-1}$ satisfies $(*')$. 
For $n \geq 1$ we have
$$\delta_{3n-1} -2\sum_{i=n+1}^\infty a_{3i} = q^{n}(\frac{10q-1}{1-q}) -2\sum_{i=n+1}^\infty 2q^{i} =q^{n}(\frac{10q-1}{1-q})  -q^n\frac{4q}{1-q} = q^n\frac{6q-1}{1-q} \geq 0,$$
which means that for any $n\in \N$ all numbers $M^n_i$ except $\delta_{3n-2}$ satisfy $(**)$. 
Thus, the assumptions of Theorem \ref{main} are satisfied, so $E(4,3,2;q)$ is a Cantorval.

\end{example}

\begin{example} \label{Palis}
Consider the set $E(4,3,2;q)$, where $q = \frac{\sqrt{57}-5}{16}$. Then $q \in (\frac18, \frac16)$ which is the "mysterious" interval.
First, we will show that $k_n = 3n$ (and $m_n=1$). 
Indeed, for $n\in \N$ we have 
$$r_{3n}-a_{3n}=q^{n-1}(r_3-a_3) = q^n(\frac{9q}{1-q}-2)\approx -0,57q^{n}<0;$$
$$\delta_{3n-1}:=r_{3n-1}-a_{3n-1}=q^{n-1}(r_2-a_2)=q^{n}(\frac{9q}{1-q}+2-3)=q^n\frac{10q-1}{1-q}\approx 0,43q^{n} > 0,$$
$$\delta_{3n-2}:=r_{3n-2}-a_{3n-2}= q^{n-1}(r_1-a_1) = q^{n} (\frac{9q}{1-q}+5-4)=q^n\frac{8q+1}{1-q}\approx 2,43q^{n}>0.$$

We will show that the assumptions of Theorem \ref{main} are satisfied. 

Let $n \in \N$. We have 
$$\delta_{3n-2}-a_{3n} = q^{n-1}(\delta_1-a_3)=q^{n-1}(\frac{8q+1}{1-q}-2)=q^n\frac{10q-1}{1-q}=\delta_{3n-1} >0$$
and 
$$2r_{3n}-\delta_{3n-2} =q^{n-1}(2r_3-\delta_1)=q^n\frac{18q-8q-1}{1-q}=\delta_{3n-1}.$$
So, condition $(*)$ is satisfied for $\delta_{3n-2}$. Since $\delta_{3n-1} >0$, $(*')$ also holds.
If $n > 1$, then we know that
$2r_{3n-3}-\delta_{3n-5} =\delta_{3n-5}-a_{3n-3}=\delta_{3n-4}$,
but we also have 
$$\delta
_{3n-4} - 2a_{3n} = q^{n-1}(\frac{10q-1}{1-q}-4q)=q^{n-1}\frac{4q^2+6q-1}{1-q}\approx 0,12 q^{n-1}> 0,$$
so number $\delta_{3n-4}$ satisfies $(**)$.
Moreover,
if $n > 2$, then
$$\delta_{3n-7} - 2a_{3n-3}-\delta_{3n-2} = q^{n-2}\frac{4q^2+6q-1}{1-q}-q^n\frac{8q+1}{1-q}=q^{n-2}\frac{-8q^3+3q^2+6q-1}{1-q} = 0,$$
so
$\delta_{3n-7} - 2a_{3n-3} =\delta_{3n-2}$. 
Therefore,
all numbers $M^n_i$ are equal to either $\delta_{3n-2}$ or $\delta_{3n-4}$.
So, condition $(*)$ or $(**)$ is satisfied for each $M^n_i$. Thus, by Theorem \ref{main}, $E(a_n)$ is a Cantorval.

\end{example}
It is worth pointing out that the above example is the first known multigeometric Cantorval such that $q \neq q_1$ and $q \notin [q_2,q_3)$, where $q_1,q_2,q_3$ are such as in \cite{BFS} and \cite{BBFS}. 

It is a quite natural question if the Palis-like conjecture for achievement sets of multigeometric sequences is true, that is, if every such a set either has the Lebesgue measure zero or contains an interval. For all known examples this is the case.
Example \ref{Palis} may be a step forward towards proving that this conjecture is true.

Looking on Examples \ref{18} and \ref{przedz} one may wonder if Theorem \ref{main} generalizes results from \cite{BFS} (at least for decreasing sequences). It turned out that if $q \in [q_2,q_3)$, then the assumptions of Theorem \ref{main} are not always satisfied. This is because, in Theorem \ref{main} we always assume that a gap is "covered" by a single overlap. But this is not always the case for multigeometric sequences with $q \in [q_2,q_3)$. Sometimes, single gap is covered by many overlaps although the geometric idea of this covering is very similar as in Theorem \ref{main} (it is the case for example for $E(32,17,16,8,4,2;q)$ from the paper \cite{MM}). It suggests that this theorem can be naturally generalized, by allowing more than one overlap to cover one gap. Another possible direction of generalization would be to consider sequences which are nonincreasing but with repeating terms. Similar idea was for example used to prove Theorem 4.19. in \cite{MNP} or main theorems in \cite{FN23} and \cite{Now2}. However, it seems that the assumptions and proof would be even more complicated and technical.

\end{document}